\documentclass[12pt, reqno]{amsart}
\usepackage{amsmath, amsthm, amscd, amsfonts, amssymb, mathtools, color, hyperref}
\usepackage[dvipsnames]{xcolor}

\newtheorem{theorem}{Theorem}[section]
\newtheorem{lemma}[theorem]{Lemma}
\newtheorem{proposition}[theorem]{Proposition}

\theoremstyle{definition}

\newtheorem{example}[theorem]{Example}

\theoremstyle{remark}
\newtheorem{remark}[theorem]{Remark}
\numberwithin{equation}{section}
\newcommand*{\bigchi}{\mbox{\Large$\chi$}}

\begin{document}
	\setcounter{page}{1}
	
	\title[Hoffman-Wielandt inequality]
	{The Hoffman-Wielandt inequality for quaternion matrices and quaternion matrix 
		polynomials}
	
	\author[Pallavi]{Pallavi Basavaraju}
	\address{Pallavi Basavaraju\\
		 Department of Mathematics\\
		Dr. G. Shankar Government Women's First Grade College and P.G. Study Centre\\
		Ajjarakadu, Udupi, Karnataka -- 576101, Karnataka, India}
	\email{pallavipoorna20@iisertvm.ac.in, pallavipoorna6@gmail.com}

	\author[Shrinath]{Shrinath Hadimani}
	\address{Shrinath Hadimani\\
		 Department of Mathematics\\
		Manipal Institute of Technology, Manipal 
		Academy of Higher Education\\
		Manipal -- 576104, Karnataka, India}
	\email{srinathsh3320@iisertvm.ac.in, shrinath.hadimani@manipal.edu}

	\author[Sachindranath]{Sachindranath Jayaraman}
	\address{Sachindranath Jayaraman\\
		School of Mathematics\\ 
		Indian Institute of Science Education and Research Thiruvananthapuram\\ 
		Maruthamala P.O., Vithura, Thiruvananthapuram -- 695 551, Kerala, India.}
	\email{sachindranathj@iisertvm.ac.in, 
		sachindranathj@gmail.com}

	
	
	\begin{abstract}
		The purpose of this paper is to derive the Hoffman-Wielandt inequality and its 
		generalization for quaternion matrices. Diagonalizability of the block companion matrix 
		of certain quadratic (linear) quaternion matrix polynomials is brought out. As a 
		consequence, we prove that if $Q(\lambda)$ is another quadratic (linear) quaternion 
		matrix polynomial, then under certain conditions on the coefficients, a generalization of the 
		Hoffman-Wielandt inequality for their corresponding block companion matrices holds. We also prove 
		that if $P(\lambda)$ is a quaternion matrix polynomial with unitary coefficients, then any right 
		eigenvalue $\lambda_0$ of $P(\lambda)$ lies in the annular region $\frac{1}{2} < |\lambda_0| < 2$.
	\end{abstract}
	
	\maketitle
	
	
	
	\section{Introduction}\label{sec-1}
	
	Quaternion matrices and quaternion matrix polynomials have been of considerable interest 
	to researchers in the last few years. Noncommutativity of quaternion multiplication makes 
	analysis over the quaternions quite intriguing. A good reference for quaternion matrices 
	is the survey article by Zhang \cite{Zhang} and the references cited therein. On the other 
	hand, literature on quaternion matrix polynomials is quite limited. Recently in 
	\cite{Ahmad-Ali} and \cite{Ali}, the authors discuss interesting techniques to derive 
	eigenvalue bounds of quaternion matrix polynomials. 
	
	Perturbation analysis of matrices over complex numbers and their eigenvalues is an old 
	problem but less studied for quaternion matrices. One of the well known inequalities on 
	this is the Hoffman-Wielandt inequality (see Theorem \ref{Thm-H-W inequality} below). 
	The reader may refer to Theorem $6.3.5$ of \cite{Horn-Johnson} 
	for a proof. The book by Bhatia \cite{Bhatia} gives a detailed account of the spectral variation 
	problem. Possible generalizations of the Hoffman-Wielandt inequality for complex matrices exist in the literature (see for instance Theorems $2, 3, 5, 6, 7, \ \& \ 8$ of \cite{Ikramov-Nesterenko} and 
	Theorem $2$ of \cite{Sun} and Remark $3.3(2)$ of \cite{Sun-2}). Among these generalizations of 
	Theorem \ref{Thm-H-W inequality}, obtained by relaxing  normality of one or both the matrices in 
	the above theorem, we focus our attention to the one given in \cite{Sun, Sun-2}. We state this 
	below as Theorem \ref{Thm-H-W type inequality}. This was also recently studied in \cite{Pallavi-Hadimani-Jayaraman} in the context of complex matrix polynomials. 
	In \cite{Ahmad-Ali-Ivan}, the authors give a Bauer-Fike type theorem for the right eigenvalues 
	of quaternion matrices and also discuss perturbations via a block-diagonal decomposition and 
	Jordan canonical form of quaternion matrices. In a recent work, we have also obtained location and 
	perturbation results for coneigenvalues of quaternion matrices \cite{Pallavi-Hadimani-Jayaraman-2}. 
	The purpose of this paper is to derive the Hoffman-Wielandt inequality and its generalization for 
	quaternion matrices involving their right eigenvalues. We study these in the context of quaternion 
	matrix polynomials.
	
	The notations $||\cdot||_2$ and $||\cdot||_F$ denote respectively the 
	spectral norm and the Frobenius norm of complex matrices.
	
	\begin{theorem}\label{Thm-H-W inequality}
		Let $A$ and $B$ be two $n \times n$ 
		normal matrices with $\lambda_1, \ldots, \lambda_n$ and 
		$\mu_1, \ldots, \mu_n$ as their  eigenvalues respectively given in some order. Then there exists a permutation 
		$\pi$ of $\{1, \ldots, n\}$ such that 
		\begin{center}
			$\displaystyle \sum_{i=1}^{n} |\lambda_i - \mu_{\pi(i)}|^2 \leq ||A-B||^2_F $.
		\end{center}
	\end{theorem}

	\begin{theorem}\label{Thm-H-W type inequality}
		Let $A$ be a diagonalizable matrix of order $n$ and $B$ be a normal matrix of order $n$, 
		with eigenvalues $\lambda_1, \lambda_2, \ldots, \lambda_n$ and $\mu_1, \mu_2, \ldots, 
		\mu_n$, respectively. Let $X$ be a nonsingular matrix whose columns are eigenvectors of 
		$A$. Then, there exists a permutation $\pi$ of $\{1, \ldots, n\}$ such that 
		\begin{center}
			$\displaystyle \sum_{i=1}^{n} |\lambda_i - \mu_{\pi(i)}|^2 \leq 
			||X||^2_2 ||X^{-1}||^2_2 ||A-B||^2_F$.
		\end{center}
	\end{theorem}
	
	An example to illustrate the importance of normality of one of the matrices 
	in Theorem \ref{Thm-H-W type inequality} can be easily provided.

	This paper is organized as follows. Section \ref{sec-2} contains notations, 
	preliminaries and necessary results from \cite{Zhang} and \cite{Ahmad-Ali}. The 
	Hoffman-Wielandt inequality and its generalization for quaternion matrices 
	are proved in Section \ref{sec-3.1}. In Section \ref{sec-3.2}, we discuss these for block companion 
	matrices of quaternion matrix polynomials. Diagonalizability of the block companion matrix 
	of certain quadratic (linear) quaternion matrix polynomials is brought out. We also prove that if 
	$P(\lambda)$ is a quaternion matrix polynomial with unitary coefficients, then any right eigenvalue $\lambda_0$ of $P(\lambda)$ lies in the annular region $\frac{1}{2} < |\lambda_0| < 2$.

	\section{Notations and preliminaries}\label{sec-2}
	
	Throughout this paper, we use the following notation and terminology. The fields 
	of real and complex numbers are denoted by $\mathbb{R}$ and $\mathbb{C}$ 
	respectively. $\mathbb{C}^+$ denotes the closed upper half plane of the complex plane. The 
	set $\mathbb{H}:= \{a_0+ a_1 i + a_2 j+ a_3 k \ | \ a_i \in 
	\mathbb{R}\}$ with $i^2= j^2=k^2=ijk=-1$, denotes the 
	set of real quaternions. The conjugate and modulus of an element $q \in \mathbb{H}$ are 
	denoted and defined as $\bar{q}:= a_0- a_1 i -a_2 j- a_3 k$, 
	$|q|:=\sqrt{a_0^2 +a_1^2+a_2^2+a_3^2}$ respectively. Define $Re(q)= a_0$ as the real part 
	of $q$, $Co(q)=a_0+a_1 i$, the complex part of $q$ and 
	$Im(q)=a_1 i+a_2j+a_3k$, the imaginary part of $q$. Two 
	quaternions $p$ and $q$ are similar (written as $p \sim q$) if there exists a nonzero 
	quaternion $r$ such that $r^{-1}qr =p$. If $p \sim q$, then $|p|=|q|$. This similarity is 
	an equivalence relation on $\mathbb{H}$. 
	
	Let $M_n\mathbb(X)$ denote the set of all 
	$n \times n$ matrices whose entries are from $X$, where $X$ is either $\mathbb{R}$, 
	$\mathbb{C}$ or $\mathbb{H}$. Given $A = (a_{ij}) \in M_n(\mathbb{H})$, the conjugate 
	transpose of $A$ denoted by $A^*$ is defined as $A^* = (\bar{a}_{ji})$. 
	$A \in M_n(\mathbb{H})$ is said to be normal if $AA^* = A^*A$, unitary if $AA^*=I=A^*A$, 
	Hermitian if $A^*=A$; and invertible if $AB=BA=I$ for some $B \in M_n(\mathbb{H})$. 
	An $n\times n$ quaternion matrix $A$ is said to be positive (semi)definite if $A$ is 
	Hermitian and $x^*Ax > (\geq) 0$ for all nonzero vector $x \in \mathbb{H}^n$. 
	Since the set of quaternions is a noncommutative division ring, there exists a 
	notion of right and left eigenvalues for a matrix in $M_n(\mathbb{H})$. For 
	$A \in M_n(\mathbb{H})$, a quaternion $\lambda_0 \in \mathbb{H}$ is called a right (left) 
	eigenvalue if there exists a nonzero vector $x \in \mathbb{H}^n$ such that $Ax=x\lambda_0$ 
	($Ax=\lambda_0 x$). Note that if $\lambda_0$ is a right eigenvalue of $A$, then any 
	quaternion similar to $\lambda_0$ is also a right eigenvalue of $A$. Therefore, there can 
	be infinitely many right eigenvalues for $A \in M_n(\mathbb{H})$. However, any matrix 
	$A \in M_n(\mathbb{H})$ has exactly $n$ right eigenvalues, which are complex numbers having 
	nonnegative imaginary parts (Theorem $5.4$, \cite{Zhang}). These right eigenvalues are called 
	the standard eigenvalues of $A$. Any quaternion which is a right eigenvalue of $A$ lies in 
	one of the equivalence classes of these $n$ complex numbers. On the contrary, the number of 
	left eigenvalues up to equivalent classes is still unknown and a quaternion similar to a left 
	eigenvalue need not be a left eigenvalue. Thus in comparison to left eigenvalues, 
	right eigenvalues have received more attention in the literature. 
	
	$A \in M_n(\mathbb{H})$ is said to be diagonalizable if there exists an invertible matrix 
	$S \in M_n(\mathbb{H})$ such that $S^{-1}AS=J$, where 
	$J= \text{diag}(\lambda_1,\lambda_2,\dots,\lambda_n)$ and 
	$\lambda_i$'s are the standard eigenvalues of $A$. Given a matrix $A \in M_n(\mathbb{H})$, 
	we can write $A=A_1 +A_2 j$, where $A_1$, $A_2 \in M_n(\mathbb{C})$. We can then 
	associate to $A$, a $2n \times 2n$ complex block matrix, $\bigchi_A = 
	\begin{bmatrix}
		A_1 & A_2 \\
		-\bar{A_2} & \bar{A_1}.
	\end{bmatrix}$, called the complex adjoint matrix of $A$. For $A \in M_n(\mathbb{H})$, the 
	Frobenius norm and the spectral norm are defined as follows:\\
	(1) $||A||_F = \big(\text{trace}A^*A\big)^{1/2}$ \\
	(2) $||A||_2 = \displaystyle \sup_{x \neq 0} \Bigg\{\frac{||Ax||_2}{||x||_2}: 
	x \in \mathbb{H}^n \Bigg\}$.\\ For $A \in M_n(\mathbb{H})$ it is easy to verify that 
	$\sqrt{2}||A||_F =||\bigchi_A||_F$ and $||A||_2 =||\bigchi_A||_2$. Unless and until specified, 
	all the matrices considered in this paper are quaternion matrices. We now list some 
	fundamental information about matrix A and its complex adjoint matrix $\bigchi_A$. These
	results can be found in \cite{Zhang}.
	
	\begin{proposition}\label{Prop-Properties-complex adjoint}
		Let $A,B \in M_n(\mathbb{H})$ and $\alpha \in \mathbb{R}$. Then
		\begin{itemize}
			\item[(a)] $\bigchi_{I_n}= I_{2n}$.
			\item[(b)] $\bigchi_{AB} = \bigchi_A \bigchi_B$.
			\item[(c)] $\bigchi_{\alpha A} = \alpha \bigchi_A$.
			\item[(d)] $\bigchi_{A+B} = \bigchi_A +\bigchi_B$.
			\item[(e)] $\bigchi_{A^*} = (\bigchi_A)^*$.
			\item[(f)] $\bigchi_{A^{-1}} = (\bigchi_A)^{-1}$, if $A^{-1}$ exists.
			\item[(g)] $\bigchi_A$ is unitary, Hermitian, diagonalizable, invertible or normal 
			if and only if $A$ is unitary, Hermitian, diagonalizable, invertible, or normal respectively.
			\item[(h)] $\bigchi_A \bigchi_B= \bigchi_B \bigchi_A$ if and only if $AB=BA$.
			\item[(i)] $\lambda_1, \lambda_2, \dots, \lambda_n$ are the standard eigenvalues
			of $A$ if and only if $\lambda_1, \lambda_2, \dots, \lambda_n, \\
			\bar{\lambda}_1, \bar{\lambda}_2, \dots, \bar{\lambda}_n$ are the eigenvalues of $\bigchi_A$.
			\item[(j)] A complex matrix $A$ is diagonalizable over $\mathbb{H}$ if and only if it is 
			diagonalizable over $\mathbb{C}$.
		\end{itemize}	
	\end{proposition} 
	
	The following remark is worth pointing out.
	
	\begin{remark}\label{Rem-standard eigenvaues-complex matrix}
		\
		\begin{enumerate}
			\item For a complex matrix $A$, the eigenvalues over $\mathbb{C}$ and the standard 
			eigenvalues over $\mathbb{H}$ are not the same in general. For example consider 
			$A=\begin{bmatrix}
				1+i & 0 \\
				0 & 1-i
			\end{bmatrix}$. Then the eigenvalues of $A$ over $\mathbb{C}$ are $1+i$ and $1-i$, whereas,
			the standard eigenvalues of $A$ over $\mathbb{H}$ are $\lambda_1 = \lambda_2 = 1+i$.
			
			\item Note that as sets $M_n(\mathbb{C}) \subsetneq M_n(\mathbb{H})$ and every unitary 
			matrix in $M_n(\mathbb{C})$ is a unitary matrix in $M_n(\mathbb{H})$ as well. However, there are 
			more unitary matrices in $M_n(\mathbb{H})$ than in $M_n(\mathbb{C})$. 
			For example the matrix $A=\begin{bmatrix}
				j & 0 \\
				0 & j
			\end{bmatrix}$ is unitary in $M_n(\mathbb{H})$ but $A \notin M_n(\mathbb{C})$. The same 
			happens for normal, Hermitian, positive (semi)definite and diagonalizable matrices.
		\end{enumerate}
		
	\end{remark}

	We now define matrix polynomials where the coefficients are from $M_n(\mathbb{H})$. As 
	the multiplication of quaternions is noncommutative, we have the notion of right and left 
	matrix polynomials. A right quaternion matrix polynomial of size $n$ and degree $m$ is a 
	mapping $P: \mathbb{H} \rightarrow M_n(\mathbb{H})$ defined by 
	$P(\lambda)= \displaystyle \sum_{i=0}^{m}A_i \lambda^i$, where $A_i \in M_n(\mathbb{H})$ 
	(that is, the indeterminate $\lambda$ is on the right of the matrix coefficients). 
	A scalar $\lambda_0 \in \mathbb{H}$ is said to be a right eigenvalue of $P(\lambda)$ if 
	$\displaystyle \sum_{i=0}^{m}A_i x \lambda_0^i =0$ for some nonzero vector $
	x \in \mathbb{H}^n$.  If the leading coefficient $A_m$ is invertible we associate a 
	monic quaternion matrix polynomial 
	$P_U(\lambda) = I \lambda^m + B_{m-1} \lambda^{m-1}+ \cdots + B_{1}\lambda +B_0$, where 
	$B_i= A_m^{-1} A_i$ for $i= 0, 1, \ldots , m-1$. It is easy to verify that the right 
	eigenvalues of $P(\lambda)$ and $P_U(\lambda)$ are the same. The right eigenvalues of $P_U$ 
	are the same as the right eigenvalues of the block companion matrix 
	$C_P = \begin{bmatrix}
		0 & I & 0 & \cdots & 0 \\
		0 & 0 & I & \cdots & 0 \\
		\vdots & \vdots & \vdots & \ddots & \vdots \\
		-B_0 & -B_1 & -B_2 & \cdots & -B_{m-1}\\ 
	\end{bmatrix} \in M_{mn}(\mathbb{H})$ (for details, see \cite{Ahmad-Ali}). We define 
	the standard eigenvalues of $P(\lambda)$ to be the standard eigenvalues of $C_P$. 
	
	An element $\lambda_0 \in \mathbb{H}$ is called a left eigenvalue of $P(\lambda)$ if 
	$\displaystyle \sum_{i=0}^{m}A_i \lambda_0^i x =0$ for some nonzero vector 
	$x \in \mathbb{H}^n$. Note that the left eigenvalues of $P(\lambda)$ coincide with 
	the left eigenvalues of $C_P$ (the proof of this statement is very similar to that of 
	Theorem $5.1$ of \cite{Ahmad-Ali}). An $n \times n$ left quaternion matrix polynomial is 
	a map $Q: \mathbb{H} \rightarrow M_n(\mathbb{H})$ defined by 
	$Q(\lambda)= \displaystyle \sum_{i=0}^{m} \lambda^i A_i$, where $A_i \in M_n(\mathbb{H})$ 
	(note that in this case, the indeterminate $\lambda$ is on the left of the matrix 
	coefficients). A quaternion $\lambda_0$ is called a left eigenvalue of $Q(\lambda)$ if 
	$\displaystyle \sum_{i=0}^{m} \lambda_0^i A_i x =0$ for some nonzero vector 
	$x \in \mathbb{H}^n$. However, one cannot define right eigenvalues for left quaternion 
	matrix polynomials in the way we defined them for right quaternion matrix polynomials. 
	Moreover, when the leading coefficient is invertible, it is easy to verify that the left 
	eigenvalues of a left quaternion matrix polynomial are the same as the left eigenvalues 
	of the corresponding block companion matrix (for details, see \cite{Ahmad-Ali}). 
	
	As mentioned earlier, for a quaternion matrix the number of left eigenvalues is still unknown and 
	hence is less studied in the literature. Therefore, as far as this paper goes, we restrict 
	ourselves to the study of right eigenvalues of right quaternion matrix polynomials. 
	We refer to right quaternion matrix polynomials as just quaternion matrix 
	polynomials. 
	
	\noindent
	{\it Assumptions:} Throughout this paper, we assume that the leading coefficient of a 
	quaternion matrix polynomial is invertible.
	
	\section{Main results}\label{sec-3}
	
	The key results of this paper are presented in this section, which has been further 
	divided into subsections for reading convenience.
	
	\subsection{The Hoffman-Wielandt inequality and its generalization for quaternion matrices} 
	\hspace*{\fill}\label{sec-3.1}
	
	The Hoffman-Wielandt inequality in general does not hold for quaternion matrices if we 
	consider right eigenvalues that are not standard. For example consider normal matrices 
	$A = 
	\begin{bmatrix}
		1+i & 0\\
		0 & 1
	\end{bmatrix}$ and $B = 
	\begin{bmatrix}
		i & 0 \\
		0 & 1
	\end{bmatrix}$, whose standard eigenvalues are $1+i$, $1$ and $i$, $1$ respectively. 
	For the matrix $A$ we consider the right eigenvalues $\mu_1 = 1-i$, $\mu_2 =1$ and 
	for $B$ we consider $\delta_1=i$, $\delta_2=1$. Then for any permutation $\pi$ on 
	$\{1, 2\}$ the summation $\displaystyle \sum_{i=1}^{2} |\mu_i -\delta_{\pi(i)}|^2 >1$, 
	whereas $||A-B||^2_F =1$. We therefore consider the standard right eigenvalues of 
	quaternion matrices and prove the Hoffman-Wielandt inequality for standard eigenvalues 
	below.
	
	\begin{theorem}\label{Thm-H-W inqequality-quaternions}
		Let $A$ and $B$ in $M_n(\mathbb{H})$ be normal matrices. Let $\mu_1, \mu_2 \ldots, \mu_n$ 
		and $\delta_1, \delta_2 \ldots, \delta_n$ be the standard eigenvalues of $A$ and $B$ 
		respectively. Then there exists a permutation $\pi$ on the indices $1,2, \ldots n$ such that 
		\begin{equation}\label{Eqn-H-W inequality}
			\displaystyle \sum_{i=1}^{n} |\mu_i 
			- \delta_{\pi(i)}|^2 \leq ||A-B||^2_F.
		\end{equation}
	\end{theorem}
	
	\begin{proof}
		As $A$ and $B$ are normal, we see from  
		Proposition \ref{Prop-Properties-complex adjoint} $(g)$ that the matrices $\bigchi_{A}$ and 
		$\bigchi_{B}$ are normal matrices in $M_{2n}(\mathbb{C})$. Therefore, the eigenvalues of $\bigchi_{A}$ and $\bigchi_{B}$ are 
		$\mu_1, \mu_2, \ldots,\mu_n, \bar{\mu}_1, \bar{\mu}_2, \ldots, \bar{\mu}_n$ and 
		$\delta_1, \delta_2, \ldots, \delta_n, \bar{\delta}_1, \bar{\delta}_2, \ldots, 
		\bar{\delta}_n$ respectively. Let us rename these eigenvalues as 
		$\mu_1, \mu_2, \ldots, \mu_{2n}$ and $\delta_1, \delta_2, \ldots, \delta_{2n}$ respectively 
		in the same order. Then, from Theorem \ref{Thm-H-W inequality} we infer that there exists a permutation $\sigma$ on the indices $1, 2, \ldots, 2n$ such that 
		\begin{equation*}
			\displaystyle \sum_{i=1}^{2n} |\mu_i -\delta_{\sigma(i)}|^2 \leq 
			||\bigchi_{A}-\bigchi_{B}||^2_F.
		\end{equation*} \\
		Consider,
		\begin{align*}
			||\bigchi_{A}-\bigchi_{B}||^2_F 
			&\geq \displaystyle \sum_{i=1}^{2n} |\mu_i - \delta_{\sigma(i)}|^2\\ 
			&= \displaystyle \sum_{i=1}^{n}|\mu_i-\delta_{\sigma(i)}|^2 + \displaystyle \sum_{i=n+1}^{2n}
			|\bar{\mu}_{i-n}-\delta_{\sigma(i)}|^2.
		\end{align*}
		\noindent
		Define $\gamma_i:= 
		\begin{cases}
			\delta_{\sigma(i)}, &\quad \text{if\space \space} 1 \leq \sigma(i) \leq n\\
			\bar{\delta}_{\sigma(i)}, &\quad\text{if \space \space} n+1 \leq \sigma(i) \leq 2n.
		\end{cases}$\\
		\noindent
		Recalling that for any two elements $a, b \in \mathbb{C}^+, \ |\bar{a}-b| \geq |a-b|$, 
		and $|\bar{a}-\bar{b}| = |a-b|$, we see that 
		\begin{align*}
			\displaystyle \sum_{i=1}^{n} |\mu_i - \delta_{\sigma(i)}|^2 + 
			\displaystyle \sum_{i=n+1}^{2n} |\bar{\mu}_{i-n} - \delta_{\sigma(i)}|^2 
			\geq \displaystyle \sum_{i=1}^{n} |\mu_i - \gamma_i|^2 + 
			\displaystyle \sum_{i=n+1}^{2n} |\mu_{i-n} - \gamma_i|^2. 
		\end{align*} 
		Since $2||A-B||^2_F = ||\bigchi_{A}-\bigchi_{B}||^2_F$, 
		we have, 
		\begin{equation}\label{Eqn-1}
			\displaystyle \sum_{i=1}^{n}|\mu_i-\gamma_i|^2 + \displaystyle 
			\sum_{i=n+1}^{2n}|\mu_{i-n}-\gamma_i|^2 \leq 2||A-B||^2_F.
		\end{equation} 
		Note that $\{\gamma_1, \ldots, \gamma_{2n}\} = \{\delta_1, \ldots, \delta_n\}$, 
		where each $\delta_i$ repeats exactly twice in Inequality \eqref{Eqn-1}. \\
		Let $S_1 = \displaystyle \sum_{i=1}^{n}|\mu_i-\gamma_i|^2$ and
		$S_2= \displaystyle \sum_{i=n+1}^{2n}|\mu_{i-n}-\gamma_i|^2$. 
		
		\noindent
		{\it Claim:} It is possible to rearrange the summations 
		$S_1$ and $S_2$ by interchanging the summands in $S_1$ (if necessary) with the 
		summands in $S_2$ to get two new summations in which both $\mu_i$'s and $\gamma_i$'s 
		are distinct.
		
		\noindent
		{\it Proof of the Claim:} 
		Let $k\leq n$ be the smallest integer such that 
		$\gamma_1,\gamma_2,\dots,\gamma_k$ are distinct in $S_1$. If $k=n$, we are done. 
		Assume that $k< n$. Consider the $(k+1)^{th}$ summand, $|\mu_{k+1} -\gamma_{k+1}|^2$ of 
		$S_1$. Then $\gamma_{k+1} = \gamma_q$ for some $1\leq q \leq k$. Swap 
		$|\mu_{k+1} -\gamma_{k+1}|^2$ with the $(k+1)^{th}$ summand $|\mu_{k+1} -\gamma_{n+k+1}|^2$ 
		of $S_2$. Then, the first $(k+1)$ summands of $S_1$ is 
		\begin{equation*}
			\displaystyle \sum_{i=1}^{k}|\mu_i-\gamma_i|^2+|\mu_{k+1}-\gamma_{n+k+1}|^2.
		\end{equation*} 
		If $\gamma_{n+k+1} \neq \gamma_r$ for any $1 \leq r \leq k$, we stop here. 
		If not, then $\gamma_{n+k+1} = \gamma_r$ for some $1\leq r \leq k$. Note that $r \neq q$, 
		because each $\gamma_i$ repeats exactly twice and we already have $\gamma_{k+1} = \gamma_q$. 
		Now swap the summand $|\mu_r-\gamma_r|^2$ of $S_1$ with the $r^{th}$ summand, 
		$|\mu_r-\gamma_{n+r}|^2$ of $S_2$ and proceed this procedure. After $(k+1)$ repetitions, 
		the first $(k+1)$ terms of $S_1$ is 
		\begin{equation*}
			\displaystyle \sum_{i=1}^{k+1} |\mu_i - \gamma_{j_i}|^2, 
		\end{equation*} 
		where $\gamma_{j_s} \neq \gamma_{j_t}$ if $s \neq t$. Thus, by mathematical induction, 
		we can rearrange the summands of $S_1$ and $S_2$ such that 
		\begin{equation*}
			S_1 +S_2 = 
			\displaystyle \sum_{i=1}^{n}|\mu_i-\gamma_{j_i}|^2 + 
			\displaystyle \sum_{i=1}^{n}|\mu_{i}-\gamma_{k_i}|^2, 
		\end{equation*} 
		where $\{\gamma_{j_1},\gamma_{j_2},\dots,\gamma_{j_n}\} = 
		\{\gamma_{k_1}, \gamma_{k_2},\dots,\gamma_{k_n}\} = \{\delta_1,\delta_2,\dots,\delta_n\}$. 
		This proves the claim.
		
		Define permutations $\sigma_1$ and $\sigma_2$ on $\{1, 2, \ldots, n\}$ by \\
		\begin{align*}
			\sigma_1(i) = j_i  \ \text{and} \ 
			\sigma_2(i) = k_i.
		\end{align*} 
		Then, 
		\begin{align*}
			\displaystyle \sum_{i=1}^{n}|\mu_i-\delta_{\sigma_1(i)}|^2 + 
			\displaystyle \sum_{i=1}^{n}|\mu_{i}-\delta_{\sigma_2(i)}|^2 
			& = \displaystyle \sum_{i=1}^{n}|\mu_i-\gamma_{j_i}|^2 + 
			\displaystyle \sum_{i=1}^{n}|\mu_{i}-\gamma_{k_i}|^2 \leq 2||A-B||^2_F.
		\end{align*} 
		Assuming without loss of generality that 
		\begin{equation*}
			\displaystyle \sum_{i=1}^{n}|\mu_i-\delta_{\sigma_1(i)}|^2 \leq 
			\displaystyle \sum_{i=1}^{n}|\mu_i-\delta_{\sigma_2(i)}|^2, 
		\end{equation*} 
		we thus conclude that 
		\begin{equation*}
			2 \displaystyle \sum_{i=1}^{n}|\mu_i-\delta_{\sigma_1(i)}|^2 \leq 2||A-B||^2_F.
		\end{equation*} 
		This implies 
		\begin{equation*}
			\displaystyle \sum_{i=1}^{n}|\mu_i-\delta_{\pi(i)}|^2 \leq ||A-B||^2_F, 
		\end{equation*} 
		where $\pi=\sigma_1$.
	\end{proof}
	
	We now prove a generalization of the Hoffman-Wielandt inequality along the one stated in 
	Theorem \ref{Thm-H-W type inequality}.
	
	\begin{theorem}\label{Thm-H-W type inequality_quaternions}
		Let $A$ be a diagonalizable matrix in $M_n(\mathbb{H})$ and $B$ be a normal matrix in
		$M_n(\mathbb{H})$. Let $\mu_1,\mu_2,\dots,\mu_n$ and $\delta_1, \delta_2, \ldots, \delta_n$
		be the standard eigenvalues of $A$ and $B$ respectively. Let $X$ be a nonsingular matrix such
		that $X^{-1}AX=D=\text{diag}(\mu_1,\mu_2,\dots,\mu_n)$. Then there exists a 
		permutation $\pi$ on $\{1,2,\dots,n\}$ such that 
		\begin{center}
			$\displaystyle \sum_{i=1}^{n}|\mu_i-\delta_{\pi(i)}|^2\leq ||X||^2_2||X^{-1}||^2_2
			||A-B||^2_F$.
		\end{center}
	\end{theorem}
	
	\begin{proof} 
		Note that the eigenvalues of $\bigchi_A$ and $\bigchi_B$ are 
		$\mu_1, \mu_2, \ldots,\mu_n, \bar{\mu}_1, \bar{\mu}_2, \ldots, \bar{\mu}_n$ and 
		$\delta_1, \delta_2, \ldots, \delta_n, \bar{\delta}_1, \bar{\delta}_2, \ldots, 
		\bar{\delta}_n$ respectively. 
		Let us rename these as $\mu_1, \mu_2,\dots,\mu_{2n}$ and $\delta_1, \delta_2,
		\dots,\delta_{2n}$ respectively, in the same order. Notice that $\bigchi_B$ is normal as $B$ is 
		normal. Consider the matrix, 
		$\bigchi_{X^{-1}AX}=\bigchi_D=
		\begin{bmatrix}
			D & 0 \\
			0 & \bar{D}
		\end{bmatrix}$. Since 
		\begin{center}
			$\bigchi_{X^{-1}AX} = \bigchi_{X^{-1}}\bigchi_A\bigchi_X = 
			(\bigchi_X)^{-1}\bigchi_A\bigchi_X$,
		\end{center} we have 
		\begin{equation*}
			(\bigchi_X)^{-1}\bigchi_A\bigchi_X = 
			\text{diag}(\mu_1, \mu_2, \ldots, \mu_n, \bar{\mu}_1, \bar{\mu}_2, \ldots, \bar{\mu}_n). 
		\end{equation*}
		This implies that $\bigchi_A$ is diagonalizable over $\mathbb{C}$ through 
		$\bigchi_X$. Therefore by Theorem \ref{Thm-H-W type inequality} there exists a permutation 
		$\sigma$ on the indices $1, 2, \ldots, 2n$ such that 
		$\displaystyle \sum_{i=1}^{2n}|\mu_i-\delta_{\sigma (i)}|^2 \leq 
		||\bigchi_X||^2_2 ||\bigchi_X^{-1}||^2_2||\bigchi_A-\bigchi_B||^2_F$. 
		Since $||\bigchi_X||_2=||X||_2$, $||\bigchi_X^{-1}||_2=||X^{-1}||_2$ and 
		$2||A-B||^2_F = ||\bigchi_{A}-\bigchi_{B}||^2_F$, this implies 
		$\displaystyle \sum_{i=1}^{2n}|\mu_i-\delta_{\sigma(i)}|^2 \leq 
		2||X||^2_2||X^{-1}||^2_2||A-B||^2_F$. As in the proof of the Theorem 
		\ref{Thm-H-W inqequality-quaternions}, there exists a permutation $\pi$ on $\{1,2,\ldots,n\}$ 
		such that $\displaystyle 2\displaystyle \sum_{i=1}^{n}|\mu_i-\delta_{\pi(i)}|^2 \leq 
		\displaystyle \sum_{i=1}^{2n}|\mu_i-\delta_{\sigma(i)}|^2$. Therefore
		$\displaystyle \sum_{i=1}^{n}|\mu_i-\delta_{\pi(i)}|^2\leq ||X||^2_2||X^{-1}||^2_2||A-B||^2_F$. 
	\end{proof}
	
	\subsection{The Hoffman-Wielandt inequality for quaternion matrix polynomials}
	\hspace*{\fill}
	\label{sec-3.2}
	
	In this section, we investigate the Hoffman-Wielandt and its generalization for block companion matrices of 
	quaternion matrix polynomials. We begin with the following lemma, whose proof is a routine computation.
	
	\begin{lemma}\label{Lem-normal condition}
		Let $P(\lambda) = I \lambda^m + A_{m-1}\lambda^{m-1} + \cdots + A_1\lambda + A_0$ be a monic 
		matrix polynomial with block companion matrix $C_P$. 
		
		\begin{itemize}
			\item[(i)] If $m=1$, then $C_P$ is normal if and only if $A_0$ is normal.
			\item[(ii)] If $m=2$, then $C_P$ is normal if and only if $A_0$ is unitary, $A_1$ is normal 
			and $A_1^*A_0 = - A_1$.
			\item[(iii)] If $m \geq 3$, then $C_P$ is normal if and only if $A_0$ is unitary and 
			$A_1 = \cdots = A_{m-1} =0$.
		\end{itemize}
	\end{lemma}
	
	The inequality \eqref{Eqn-H-W inequality} does not generally hold for block 
	companion matrices of linear quaternion matrix polynomials whose coefficients are normal matrices. 
	We illustrate this with an example taken from \cite{Pallavi-Hadimani-Jayaraman}. 
	
	\begin{example}
		Let $P(\lambda) = \begin{bmatrix}
			2 & 0 \\
			0 & -2
		\end{bmatrix} \lambda +                                                                                                                                         \begin{bmatrix}
			2 & 2 \\
			2 & -14
		\end{bmatrix}$,
		$Q(\lambda) = \begin{bmatrix}
			1 & 0 \\
			0 & -\frac{5}{4}
		\end{bmatrix}  \lambda + \begin{bmatrix}
			2 & 5 \\
			5 & -\frac{30}{4}
		\end{bmatrix}$ whose respective block companion matrices are $C_P=\begin{bmatrix}
			-1 & -1 \\
			1 & -7
		\end{bmatrix}$ and $C_Q = \begin{bmatrix}
			-2 & -5 \\
			4 & -6
		\end{bmatrix}$. The standard eigenvalues of $C_P$ and $C_Q$ are respectively
		$\lambda_1 = - 4 - 2\sqrt{2}$, $\lambda_2 = - 4 + 2 \sqrt{2}$ and $\mu_1 = - 4 + 4i$, 
		$\mu_2 = -4 + 4i$. It is easy to verify that $||C_P - C_Q||_F^2 = 27$, whereas for any permutation 
		$\pi$ on $\{1,2\}, \ \displaystyle \sum_{i=1}^{2}| \lambda_{i} - \mu_{\pi(i)} |^2 = 48$.
	\end{example}
	
	However, for linear quaternion matrix polynomials with unitary coefficients, the inequality 
	\eqref{Eqn-H-W inequality} for the corresponding block companion matrices follows from  
	Theorem \ref{Thm-H-W inqequality-quaternions}, because the respective block companion 
	matrices are unitary. For quadratic quaternion matrix polynomials whose coefficients are 
	either normal or unitary matrices, the inequality \eqref{Eqn-H-W inequality} in general 
	fails to hold. The following example from \cite{Pallavi-Hadimani-Jayaraman} illustrates this.
	
	\begin{example}
		Let $P(\lambda) = 
		\begin{bmatrix}
			1 & 0 \\
			0 & 1 
		\end{bmatrix} \lambda^2 + 
		\begin{bmatrix}
			\frac{1}{\sqrt{2}} & \frac{1}{\sqrt{2}}\\
			\frac{1}{\sqrt{2}} & -\frac{1}{\sqrt{2}}
		\end{bmatrix} \lambda + 
		\begin{bmatrix}
			\frac{4}{\sqrt{41}} & \frac{5}{\sqrt{41}}\\
			\frac{5}{\sqrt{41}} & -\frac{4}{\sqrt{41}}
		\end{bmatrix}$ and 
		$Q(\lambda) = \begin{bmatrix}
			1 & 0 \\
			0 & 1 \\
		\end{bmatrix} \lambda^2 + \begin{bmatrix}
			\frac{1}{\sqrt{2}} & \frac{1}{\sqrt{2}} \\
			\frac{1}{\sqrt{2}} & -\frac{1}{\sqrt{2}} \\
		\end{bmatrix} \lambda + \begin{bmatrix}
			-\frac{1}{2} & \frac{\sqrt{3}}{2} \\
			-\frac{\sqrt{3}}{2} & -\frac{1}{2} 
		\end{bmatrix}$. Consider the respective block companion matrices $C_P$ and $C_Q$. The 
		standard eigenvalues of $C_P$ and $C_Q$ are $\lambda_1=1.6163$, $\lambda_2 =
		-0.4969+0.8643i$, $\lambda_3= -0.4969+0.8643i$, $\lambda_4= -0.6225$ and $\mu_1=1$, $\mu_2=1$, 
		$\mu_3=-1$, $\mu_4=-1$ respectively. Note that $||C - D||_F^2 = 4$. However,
		$\displaystyle \sum_{i=1}^{4} |\mu_{\pi(i)} - \lambda_i|^2 \geq 4.5102 > 4$ for any 
		permutation $\pi$ on $\{1,2,3,4\}$.
	\end{example}
	
	We now derive, as a consequence of Lemma \ref{Lem-normal condition} and Theorem \ref{Thm-H-W inqequality-quaternions}, the Hoffman-Wielandt inequality for block companion matrices of certain types of quaternion matrix polynomials. 
	
	\begin{theorem}
		Let $P(\lambda)$ and $Q(\lambda)$ be monic matrix polynomials of degree $m$ and 
		size $n$ which satisfy the conditions of Lemma \ref{Lem-normal condition}. If $C_P$ and 
		$C_Q$ are the block companion matrices of $P(\lambda)$ and $Q(\lambda)$, then there exists 
		a permutation $\pi$ on $\{1,2, \ldots, mn\}$ such that 
		\begin{equation}
			\displaystyle \sum_{i=1}^{mn} |\lambda_i - \mu_{\pi(i)}|^2 \leq ||C_P-C_Q||^2_F, 
		\end{equation}
		where $\{ \lambda_i \}$ and $\{\mu_i\}$ are the standard eigenvalues of $C_P$ and $C_Q$ respectively.
	\end{theorem}
	
	\begin{proof}
		By assumptions on $P(\lambda)$ and $Q(\lambda)$, the matrices $C_P$ and $C_Q$ are normal. The 
		desired conclusion follows from Theorem \ref{Thm-H-W inqequality-quaternions}.
	\end{proof}
	
	We now investigate a generalization of the Hoffman-Wielandt inequality for the 
	block companion matrices of quaternion matrix polynomials, along similar lines as stated in Theorem 
	\ref{Thm-H-W type inequality}. 
	
	Given an $n \times n$ quaternion matrix polynomial $P(\lambda)=\displaystyle 
	\sum_{i=0}^{m} A_i\lambda^i$ we associate a $2n \times 2n$ complex matrix polynomial 
	$P_{\chi}: \mathbb{C} \rightarrow M_{2n}(\mathbb{C})$ defined by 
	$P_\chi(\lambda)=\displaystyle \sum_{i=0}^{m} \bigchi_{A_i}\lambda^i$. 
	We call $P_\chi(\lambda)$ as the complex adjoint matrix polynomial of $P(\lambda)$. 
	We have the following relation between the standard eigenvalues of $P(\lambda)$ and the 
	eigenvalues of $P_{\chi}(\lambda)$.
	
	\begin{theorem}\label{Thm-eigenvalue relation}
		Let $P(\lambda) = I\lambda^m + A_{m-1}\lambda^{m-1} + \cdots + A_1\lambda + A_0$ be a 
		quaternion matrix polynomial. Then, $\lambda_1, \lambda_2, \ldots, \lambda_{mn}$ are 
		standard eigenvalues of $P(\lambda)$ if and only if 
		$\lambda_1, \lambda_2, \ldots, \lambda_{mn}, \bar{\lambda}_1, \bar{\lambda}_2, \ldots, 
		\bar{\lambda}_{mn}$ are eigenvalues of $P_\chi(\lambda)$.
	\end{theorem}
	
	\begin{proof}	
		Consider $C_P = \begin{bmatrix}
			0 & I & 0 & \cdots & 0 \\
			\vdots & \vdots & \vdots & \ddots & \vdots \\
			0 & 0 & 0 & \cdots & I\\
			-A_0 & -A_1 & -A_2 & \cdots & -A_{m-1}
		\end{bmatrix}$, the block companion matrix corresponding to $P(\lambda)$. Let 
		$A_i = A_{i1}+A_{i2} j$, where $A_{i1}, A_{i2} \in M_n(\mathbb{C})$ for 
		$i=0,1, \dots, m-1$. Then 
		\begin{center}
			$	C_P = 
			\begin{bmatrix}
				0 & I & 0 & \cdots & 0 \\
				\vdots  & \vdots & \vdots & \ddots & \vdots \\
				0 & 0 & 0 & \cdots & I\\
				-A_{01} & -A_{11}  & -A_{21} & \cdots & -A_{(m-1)1}
			\end{bmatrix}$ + $\begin{bmatrix}
				0 & 0 & 0 & \cdots & 0 \\
				\vdots & \vdots & \vdots & \ddots & \vdots \\
				0 & 0 & 0 & \cdots & 0 \\
				-A_{02} & -A_{12}  & -A_{22} & \cdots & -A_{(m-1)2}
			\end{bmatrix}j $
		\end{center}
		$	= C_{P_1}+C_{P_2}j$. 
		Therefore the complex adjoint 
		matrix of $C_P$ is 
		$\bigchi_{C_P} = 
		\begin{bmatrix}
			C_{P_1} & C_{P_2}\\
			-\bar{C}_{P_2} & \bar{C}_{P_1}
		\end{bmatrix}$. The corresponding block companion matrix of $P_\chi(\lambda)$ is 
		\begin{center}
			$C_{P_\chi} 
			= \begin{bmatrix}
				0 & I & 0 & \cdots & 0 \\
				\vdots & \vdots & \vdots & \ddots & \vdots \\
				0 & 0 & 0 & \cdots & I\\
				-\bigchi_{A_0} & -\bigchi_{A_1} & -\bigchi_{A_2} & \cdots & -\bigchi_{A_{m-1}}
			\end{bmatrix}$. 
		\end{center}
		It is easy to verify that $\bigchi_{C_P} = PC_{P_\chi}P^{-1}$, where 
		$P$ is the permutation matrix given by 
		$P = E_{11} + E_{23} + E_{35} + \cdots + E_{m(2m-1)} + E_{(m+1)2} + E_{(m+2)4} + 
		\cdots + E_{(2m)(2m)}$, where $E_{ij}$ is a block matrix of size $2mn \times 2mn$ 
		with each block of size $n \times n$ such that the $ij^{th}$ block is $I_n$ 
		(the identity matrix of size $n$) and the remaining blocks being 
		zeros for $1 \leq i,j \leq 2m$. The proof now follows from Proposition 
		\ref{Prop-Properties-complex adjoint} $(i)$.
	\end{proof}
	
	We have the following result as a consequence of the above theorem.
	
	\begin{theorem}\label{Thm-eigenvalue location}
		Let $P(\lambda)$ be a quaternion matrix polynomial all of whose coefficients are 
		unitary matrices and let $\lambda_0$ be one of its right eigenvalues. 
		Then $\frac{1}{2}<|\lambda_0|<2$.
	\end{theorem}
	
	\begin{proof}
		Since the coefficients of $P(\lambda)$ are quaternion unitary matrices, the coefficients 
		of its complex adjoint matrix polynomial, $P_{\chi}(\lambda)$ are complex unitary matrices. 
		Therefore by Theorem $3.2$ of \cite{Cameron}, if $\mu_0$ is an eigenvalue of 
		$P_{\chi}(\lambda)$, then $\frac{1}{2} < |\mu_0| < 2$. Now the result follows from the 
		previous theorem and Corollary $6.1$ from \cite{Zhang}. 	
	\end{proof}
	
	We now examine diagonalizability of the block companion matrices of certain classes of 
	quaternion matrix polynomials. We begin with linear quaternion matrix polynomials. 
	
	\begin{theorem}\label{Thm-linear case}
		Let $P(\lambda)$ be a linear quaternion matrix polynomial whose coefficients are either 
		$(a)$ unitary matrices or $(b)$ diagonal matrices or $(c)$ positive definite matrices. 
		Then the corresponding block companion matrix $C_P$ of $P(\lambda)$ is diagonalizable. 
	\end{theorem}
	
	\begin{proof}
		Let $P(\lambda) = A_1\lambda+A_0$. If the coefficients of $P(\lambda)$ are either 
		unitary or diagonal matrices then the result follows trivially. Suppose $A_1$ and $A_0$ 
		are positive definite matrices. Consider the corresponding complex adjoint matrix 
		polynomial $P_{\chi}(\lambda) = \bigchi_{A_1}\lambda+\bigchi_{A_0}$. Observe that 
		$\bigchi_{A_1}$ and $\bigchi_{A_0}$ are positive definite complex matrices. 
		Therefore by part $(3)$ of Theorem $2.2$ of \cite{Pallavi-Hadimani-Jayaraman}, the 
		block companion matrix $C_{P_{\chi}} = - \bigchi_{A_1}^{-1}\bigchi_{A_0}$ is 
		diagonalizable. Since $\bigchi_{C_P}$ is similar to $C_{P_{\chi}}$, it follows that 
		$\bigchi_{C_P}$ is diagonalizable. Hence it follows from  Proposition 
		\ref{Prop-Properties-complex adjoint} $(g)$ that $C_P= - A_1^{-1}A_0$ is diagonalizable.
	\end{proof}

	\begin{remark}\label{Rem-counter example-linear}
		When the coefficients of a linear quaternion matrix polynomial $P(\lambda)$ are either 
		normal matrices or upper(lower) triangular matrices, the corresponding block 
		companion matrix $C_P$ is not diagonalizable in general. The examples given in Remark $2.3 (1)$ 
		of \cite{Pallavi-Hadimani-Jayaraman} along with 
		Proposition \ref{Prop-Properties-complex adjoint} $(j)$ serve the purpose.
	\end{remark}

	We now consider quadratic quaternion matrix polynomials.
	
	\begin{theorem}\label{Thm-diagonalization-block companion matrix}
		Let $P(\lambda)=I\lambda^2 + U_1 \lambda + U_0$ be an $n\times n$ quaternion matrix polynomial
		where the coefficients $U_0$ and $U_1$ are commuting unitary matrices. Then the
		corresponding block companion matrix $C_P$ of $P(\lambda)$ is diagonalizable.
	\end{theorem}
	
	\begin{proof}
		Let $\bigchi_{U_1}$ and $\bigchi_{U_0}$ be the complex adjoint matrices of $U_1$ and 
		$U_0$ respectively. Consider the corresponding complex adjoint matrix polynomial of 
		$P(\lambda)$, given by $P_{\chi}(\lambda) = I\lambda^2+\bigchi_{U_1}\lambda+\bigchi_{U_0}$ 
		with the block companion matrix $C_{P_\chi}$. Since $U_1$ and $U_0$ are commuting unitary 
		matrices, the matrices $\bigchi_{U_1}$ and $\bigchi_{U_0}$ are commuting complex unitary 
		matrices. Then by Theorem $2.4$ of \cite{Pallavi-Hadimani-Jayaraman}, $C_{P_\chi}$ is 
		diagonalizable. Since $\bigchi_{C_P}$ and $C_{P_\chi}$ are similar, $\bigchi_{C_P}$ is 
		diagonalizable. Proposition \ref{Prop-Properties-complex adjoint} $(g)$ then implies that 
		$C_P$ is diagonalizable.
	\end{proof}

	\begin{remark}\label{Rem-counter examples-unitary}
		\
		\begin{enumerate}
			\item Note that if either the commutativity condition is removed or if 
			$\text{deg} \ P(\lambda) > 2$, then Theorem \ref{Thm-diagonalization-block companion matrix} 
			is not true in general. The examples given in Remark $2.5$ of \cite{Pallavi-Hadimani-Jayaraman} 
			illustrate this.
			
			\item The term $\kappa(X)=||X||_2||X^{-1}||_2$, that appears in the general form of the Hoffman-Wielandt
			inequality, is the spectral condition number of a square matrix $X$.
			In Theorem \ref{Thm-diagonalization-block companion matrix}, though  we know that 
			$\bigchi_{C_P}$ is diagonalizable through a matrix $X$ for which $\kappa(X) <2$ (see Section 
			$2.4.1$ of \cite{Pallavi-Hadimani-Jayaraman} for details), we do not know the matrix which 
			diagonalizes $C_P$. Similarly, in part (c) of Theorem \ref{Thm-linear case}, the matrix 
			$\bigchi_{C_P}$ is diagonalizable through a positive definite matrix, but we do not 
			know the matrix which diagonalizes the block companion matrix $C_P$. Hence it is difficult 
			to estimate the condition number for these matrices. The block companion matrices, in parts 
			(a) and (b) of Theorem \ref{Thm-linear case} are diagonalizable
			through a unitary matrix and identity matrix respectively. The condition number
			of these matrices is $1$.  
		\end{enumerate}
		
	\end{remark}
	
	We end this paper with the proof of the general form of the Hoffman-Wielandt inequality for block 
	companion matrices of quaternion matrix polynomials.
	
	\begin{theorem}\label{Thm-H-W inequality-matrix polynomials}
		Let $P(\lambda)$ and $Q(\lambda)$ be quadratic quaternion matrix polynomials of the same 
		size. Let $C_P$ and $C_Q$ be the corresponding block companion matrices. If the coefficients 
		of $P(\lambda)$ are commuting unitary matrices and $Q(\lambda)$ satisfies the condition of Lemma 
		\ref{Lem-normal condition} (ii), then there exists a permutation $\pi$ of the 
		indices $1, 2, \ldots, 2n$ such that 
		\begin{center}
			$\displaystyle \sum_{i=1}^{2n} |\lambda_i - \mu_{\pi(i)}|^2 
			\leq ||X||^2_2 ||X^{-1}||^2_2 ||C_P - C_Q||^2_F$
		\end{center}
		where, $\{\lambda_i\}$ and $\{\mu_i\}$ are 
		the standard eigenvalues of $C_P$ and $C_Q$ respectively, and $X$ is a nonsingular matrix.
	\end{theorem}
	
	\begin{proof}
		By Theorem \ref{Thm-diagonalization-block companion matrix} the block companion matrix $C_P$ 
		is diagonalizable and by Lemma \ref{Lem-normal condition} (ii), the matrix $C_Q$ is normal. The result then 
		follows from Theorem \ref{Thm-H-W type inequality_quaternions}.
	\end{proof}

	\begin{remark}\label{General-HW-Linear-Matrix-Poly}
		In a similar line, one can prove the general form of the Hoffman-Wielandt inequality for the block 
		companion matrices of linear quaternion matrix polynomials. If $P(\lambda)$ and $Q(\lambda)$ are linear 
		matrix polynomials of same size, where $P(\lambda)$ satisfies either of the conditions in Theorem 
		\ref{Thm-linear case} and $Q(\lambda)$ satisfies the condition (i) of Lemma \ref{Lem-normal condition}, 
		then the general form of Hoffman-Wielandt inequality follows.
	\end{remark}

	\section{Concluding Remarks}
	
	We have derived the Hoffman-Wielandt inequality and its generalization for 
	quaternion matrices. Under certain assumptions on the coefficients, diagonalizability 
	of block companion matrices of quaternion matrix polynomials is proved. Similarly, a characterization 
	to determine when the block companion matrix of a quaternion matrix polynomial is normal is given. 
	As a consequence, a generalization of the Hoffman-Wielandt inequality for the corresponding 
	block companion matrices of such quaternion matrix polynomials is derived. In addition, bound for 
	the right eigenvalues of quaternion matrix polynomials whose coefficient matrices are unitary is 
	given. The results presented in this paper lead to some interesting questions/problems, one of which has already been pointed out in Remark $\ref{Rem-counter examples-unitary}(2)$. Another interesting 
	question concerns the spectral variation involving left eigenvalues of quaternion matrices and quaternion matrix polynomials; the paper by Huang and So \cite{Huang-So} contains some interesting results on finding/computing left eigenvalues of quaternion matrices.



	%
	%
	%
	%
	

\begin{thebibliography}{20}
		
		%
		
		
		\bibitem{Ahmad-Ali}
		{\sc Sk. S. Ahmad and I. Ali}, 
		{\it Bounds for eigenvalues of matrix polynomials over quaternion 
			division algebra}, Adv. in Appl. Clifford Algebr., {\bf 26}, 4 (2016), 1095 -- 1125. 
		
		\bibitem{Ahmad-Ali-Ivan}
		{\sc Sk. S. Ahmad, I. Ali and I. Slapni\v{c}ar}, 
		{\it Perturbation analysis of matrices over a 
			quaternion division algebra}, Electron. Trans. Numer. Anal., {\bf 54}, (2021), 128 -- 149.
		
		\bibitem{Ali}
		{\sc I. Ali}, 
		{\it Bounds for the right spectral radius of quaternionic matrices}, Ukrainian 
		Math. J., {\bf 72}, 6 (2020), 723 -- 735.
		
		\bibitem{Bhatia}
		{\sc R. Bhatia}, {\it Matrix Analysis}, Graduate Texts in Mathematics, \textbf{}169, Springer Verlag, 
		New York, 1996.
		
		\bibitem{Pallavi-Hadimani-Jayaraman}
		{\sc P. Basavaraju, S. Hadimani and S. Jayaraman}, 
		{\it Hoffman-Wielandt type inequality for block 
			companion matrices of certain matrix polynomials}, Adv. Oper. Theory, {\bf 8}, 4 (2023), Paper
		No. 65.
		
		\bibitem{Pallavi-Hadimani-Jayaraman-2}
		{\sc P. Basavaraju, S. Hadimani and S. Jayaraman}, 
		{\it On coneigenvalues of quaternion 
			matrices: location and perturbation}, Indian J. Pure Appl. Math., {\bf 56}, 3  (2025) 1144 -- 1155.
		
		\bibitem{Cameron}
		{\sc T. R. Cameron}, 
		{\it Spectral bounds for matrix polynomials with unitary coefficients}, 
		Electron. J. Linear Algebra, {\bf 30}, (2015), 585 -- 591.
		
		\bibitem{Horn-Johnson}
		{\sc R. A. Horn and C. R. Johnson}, 
		{\it Matrix Analysis}, $2^{nd}$ Edition, Cambridge University Press, 
		2012.
		
		\bibitem{Huang-So}
		{\sc L. Huang and W. So}, 
		{\it On left eigenvalues of a quaternion matrix}, Linear Algebra Appl., 
		{\bf 323}, (2001), 105 -- 116.
		
		
		\bibitem{Ikramov-Nesterenko}
		{\sc Kh. D. Ikramov and Yu. R. Nesterenko}, 
		{\it Theorems of the Hoffman-Wielandt type 
			for coneigenvalues of complex matrices}, Dokl. Math., {\bf 80}, 1 (2009), 536 -- 540.
		
		\bibitem{Sun}
		{\sc Ji. G. Sun}, 
		{\it On the perturbation of the eigenvalues of a normal matrix}, Math. Numer. Sinica, 
		{\bf 6}, 3 (1984), 334 -- 336.
		
		\bibitem{Sun-2}
		{\sc Ji. G. Sun}, 
		{\it On the variation of the spectrum of a normal matrix}, Linear Algebra Appl., 
		{\bf 246}, (1996), 215 -- 223.
		
		\bibitem{Zhang}
		{\sc F. Zhang}, 
		{\it Quaternions and matrices of quaternions}, Linear Algebra Appl.,  {\bf 251},  
		(1997), 21 -- 57. 
		
	\end{thebibliography}
\end{document}